\documentclass[12pt]{amsart}
\usepackage[margin=1.in, centering]{geometry}

\usepackage{amssymb,latexsym,amsmath,extarrows, mathrsfs, amsthm}
\usepackage{graphicx}

\usepackage[colorlinks, linkcolor=blue]{hyperref}

\usepackage{color}
\usepackage{wasysym}

\usepackage{float}

\newtheorem{theorem}{Theorem}[section]
\newtheorem{lemma}[theorem]{Lemma}

\newtheorem{corollary}[theorem]{Corollary}

\theoremstyle{remark}
\newtheorem*{remark}{\bf  \itshape  \textup{Remark}}

\allowdisplaybreaks
\title{On the distribution of $k$-free numbers on the view point of random walks}

\author{Kui Liu}
\address{School of Mathematics and Statistics, Qingdao University, 308 Ningxia Road, Shinan District, Qingdao, Shandong, China}
\email{liukui@qdu.edu.cn}

\author{Meijie Lu}
\address{School of Mathematics, Shandong University, Jinan 250100, Shandong, China}
\email{meijie.lu@hotmail.com}


\begin{document}
	
	\maketitle
	
	\begin{abstract}
		In this paper, we investigate the distribution of $k$-free numbers in a class of $\alpha$-random walks on the integer lattice $\mathbb{Z}$. In these walks, the walker starts from a non-negative integer $r$ and moves to the right by $a$ units with probability $\alpha$, or by $b$ units with probability $1-\alpha$.  For $k\geq 3$, we obtain the asymptotic proportion of $k$-free numbers in a path of such $\alpha$-random walks in almost surely sense. This provides a generalization of a classical result on the distribution of $k$-free numbers in arithmetic progressions.
	\end{abstract}
	
	\section{Introduction}
	
	Let $k\geq 2$ be an integer. We say an integer $n\geq 1$ is a $k$-free number if for each prime $p\mid n$, one has $p^k\nmid n$. Traditionally, the $2$-free numbers are said to be square-free numbers. For any $N\geq 3$, let $Q_k(N)$ be the number of $k$-free numbers not exceeding $N$. Walfisz \cite{W} showed that
	\[
	Q_k(N)=\frac{N}{\zeta(k)}+O\Big(N^{1/k}\exp\big(-ck^{-8/5}(\log N)^{3/5}(\log\log N)^{1/5}\big)\Big)
	\]
	holds for some constant $c>0$, where $\zeta$ is the Riemann zeta function and $\exp(x):=e^{x},\ x\in\mathbb{R}$. This implies that  the asymptotic proportion of $k$-free numbers in the set of natural numbers $\mathbb{N}$ is $1/\zeta(k)$, i.e.
	\[ \lim\limits_{N\rightarrow\infty}\frac{Q_k(N)}{N}=\frac{1}{\zeta(k)}.
	\]
	
	People are also interested in the distribution of $k$-free numbers in arithmetic sequences.  Let $Q_k(N;q,r)$ be the number of $k$-free numbers, satisfying the congruence $n\equiv r\bmod q$ and not exceeding $N$. Then for fixed integers $q\geq 1$ and $0\leq r\leq q-1$ such that the greatest common divisor $\gcd(r,q)$ is $k$-free, we have 
	
	\begin{align}\label{eq: k-free numbers in arithmetic progressions}
		\lim\limits_{N\rightarrow \infty}\frac{Q_k(N;q,r)}{N}=\beta_k(q,r),
	\end{align}
	where
	\[
	\beta_k(q,r)=\frac{\varphi(q)}{\gcd(r,q)\varphi\big(q/\gcd(r,q)\big)}\frac{1}{q\zeta(k)}\prod\limits_{p\mid q}\Big(1-\frac{1}{p^k}\Big)^{-1},
	\]
	$\varphi$ is the Euler's totient function and the product is taken over all primes that divide $q$. This result is contained in \cite{La} (for $k=2$) and \cite{O} (for $k>2$), and see also \cite{P,CR}.
	
	Although \eqref{eq: k-free numbers in arithmetic progressions} is classical, we can still understand it from an alternative perspective. Imagine there is a 'random' walker on $\mathbb{Z}$, who starts from the point $P_0=r$ and moves $q$ units to the right in each step. Then \eqref{eq: k-free numbers in arithmetic progressions} tells us that, among all the steps taken, the asymptotic proportion of those steps in which the walker hit a $k$-free number is $\beta_k(q,r)$. With this viewpoint, it is natural to explore the distribution of $k$-free numbers arising from true random walks.
	
	Random walks are significant subjects of study in probability theory. A random walk describes a path formed by a series of random steps taken in a given space. In this paper, we focus on a specific type of random walk, which occurs on the one-dimensional integer lattice $\mathbb{Z}$. Let $\alpha\in(0,1)$,  $a,\ b\geq 1$ be two distinct integers, and $r\geq 0$ be a non-negative integer. We define an $\alpha$-random walk with step sizes $a$ and $b$, or simply an $\alpha$-random walk as follows: the starting point is $P_0:=r$ and the subsequent points are given by the recursion
	\[
	P_i:=P_{i-1}+W_i
	\]
	for $i=1,2,\cdots$, where the jumps
	\[   W_i:=\begin{cases}a, \quad {\rm~with~ probability~} \alpha,\\
		b, \quad{\rm~with~probability~}1-\alpha
	\end{cases}
	\]
	are independent random variables. Thus, the sequence $P_0,\ P_1,\ P_2,\cdots$ represents a random walk on $\mathbb{Z}$, which starts from the initial point $P_0=r$. In each step, the walker moves $a$ units to the right with probability $\alpha$ and $b$ units upwards with probability $1-\alpha$, respectively.
	
	Let $(X_i)_{i\geq 1}$ be a sequence of Bernoulli random variables, where each $X_i$ is given by
	\[ X_i:=\begin{cases}
		1,&{\rm{if}}\ P_i\ {\rm{is}}\ k\text{-}{\rm{free}},\\
		0,\ \ &\rm{otherwise.}
	\end{cases}
	\]
	For any integer $N\geq 1$, define a random variable
	\[ \overline{S}_{k,\alpha}(N):=\frac{X_1+\cdots+X_N}{N}.
	\]
	Then $\overline{S}_k(N)$ represents the proportion of those steps in which the above random walker hit a $k$-free number in the first $N$ steps. It should be noted that $\overline{S}_{k,\alpha}(N)$ depends on $a$, $b$, and $r$, although these dependencies are not explicitly stated.
	\begin{theorem}\label{thm:S_N=}
		With the above notations, for integer $k\geq 3$, we have
		\[  \lim_{N\rightarrow\infty}\overline{S}_{k,\alpha}(N)=\theta_k(a,b,r)
		\]
		almost surely, where the constant
		\begin{align}\label{def: theta(k,a,b,r)}
			\theta_k(a,b,r):=\prod_{\substack{p\\ \gcd(a,b,p^k)\mid r}}\Big(1-\frac{\gcd(a,b,p^k)}{p^k}\Big),
		\end{align}
		which is independent on $\alpha$, and the symbol $\prod_p$ denotes taking the product over primes.
	\end{theorem}
	
	\begin{remark}
		Note that if $\gcd(a,b)=1$, then $\theta_{k}(a,b,r)=1/\zeta(k)$, which coincides with the density of $k$-free numbers in $\mathbb{N}$. Unfortunately, the method we employed appears to be invalid for the case of $k=2$.
	\end{remark}
	
	\bigskip
	\textbf{Notations.} We use $\mathbb{R}$, $\mathbb{Z}$, and $\mathbb{N}$ to denote the sets of real numbers, integers, and positive integers, respectively. The notation $\gcd(m,n)$ represents the greatest common divisor of integers $m$ and $n$. The expression $f=O(g)$ (or $f\ll g$) means that $|f|\leq Cg$ for some constant $C>0$. When the constant $C$ depends on some parameters ${\bf \rho}$, we write $f=O_{\bf \rho}(g)$ (or $f\ll_\rho g)$. Moreover, we use $\mathbb{P}$, $\mathbb{E}$, and $\mathbb{V}$ to denote probability, expectation, and variance, respectively.
	\bigskip
	
	\textbf{Acknowledgements.} The authors are supported by National Natural Science Foundation of China (NSFC, Grant No. 12071238).
	
	\section{Preliminary lemmas}
	
	The following result is Lemma 11 from \cite{FF}, which can be regard as a binomial theorem with a congruence condition.
	
	\begin{lemma}\label{lem:sum_{i=r(d^k)}=}
		
		For any integers $n,\ d,\ c$ with  $n,\ d\geq 1$ and any $\alpha\in(0,1)$,we have
		\[
		\sum_{\substack{0\leq l\leq n\\l\equiv c \bmod d}}\binom{n}{l}\alpha^l(1-\alpha)^{n-l}=\frac{1}{d}+O\Big(\frac{1}{\sqrt{\alpha(1-\alpha)n}}\Big),
		\]
		where the implied constant in the $O$-symbol is absolute.
	\end{lemma}
	
	We also need the following well-known result, which can be derived directly by partial summation.
	
	\begin{lemma}\label{lem:other estimates}
		For any $x\geq 1$, we have
		\begin{align}\label{eq:sum_{s<x}s^{beta}=}
			\sum_{1\leq n\leq x}n^{\gamma}=\frac{x^{\gamma+1}}{\gamma+1}+\begin{cases}O_{\gamma}(x^{\gamma}), &{\rm{if}}\ \gamma\geq 0,\\
				O_{\gamma}(1), &{\rm{if}}\ -1<\gamma<0.
			\end{cases}
		\end{align}
		Moreover, if $\beta<-1$, we have 
		\begin{align}\label{eq:sum_{n>D}<<}
			\sum_{n>x }n^{\beta}=O_{\beta}(x^{1+\beta}).
		\end{align}
	\end{lemma}
	
	The following consequence of Lemma \ref{lem:other estimates} will be used several times in the proof of Theorem \ref{thm:S_N=}.
	
	\begin{corollary}
		\label{cor: sums over i,j}
		For any $x\geq 2$, we have
		\begin{itemize}
			\item [(i)]$$\sum_{1\leq i<j\leq x}{i}^{-1/2}=O\big(x^{3/2}\big);
			$$
			\item [(ii)]
			$$ 
			\sum_{1\leq i<j\leq x}i{j}^{-2}=O(x);
			$$
			\item [(iii)]
			$$
			\sum_{1\leq i<j\leq x}({j-i})^{-1/2}=O\big(x^{3/2}\big).
			$$
		\end{itemize}
	\end{corollary}
	
	\begin{proof}
		For the first statement, by \eqref{eq:sum_{s<x}s^{beta}=}, we obtain
		$$
		\sum_{1\leq i<j\leq x}{i}^{-1/2}=\sum_{1\leq i\leq x}{i}^{-1/2}\sum_{i<j\leq x}1\ll x\sum_{1\leq i\leq x}{i}^{-1/2}\ll x^{3/2}.
		$$
		For the second statement, we apply \eqref{eq:sum_{n>D}<<} and obtain
		$$
		\sum_{1\leq i<j\leq x}i{j}^{-2}=\sum_{1\leq i\leq x}{i}\sum_{j>i}{j}^{-2}\ll \sum_{1\leq i\leq x}1\ll x.
		$$
		For the last statement, changing variable $t=j-i$, we have
		\[
		\sum_{1\leq i<j\leq x}({j-i})^{-1/2}=\sum_{1\leq i\leq x}\sum_{i<j\leq x}({j-i})^{-1/2}=\sum_{1\leq i\leq x}\sum_{1\leq t\leq x-i}t^{-1/2}.
		\]
		Then our desired result follows again from \eqref{eq:sum_{s<x}s^{beta}=}.
	\end{proof}
	
	The following result is Lemma 2.5 from \cite{CFF}, which is essentially the second moment method in probability theory.
	
	\begin{lemma}\label{lem: the second moment method}
		
		Let $(Y_i)_{i\geq1}$ be a sequence of uniformly bounded random variables such that
		\[
		\lim_{N\rightarrow\infty}\mathbb{E}(\overline{R}_N)=\sigma,
		\]
		where
		$$
		\overline{R}_N=\frac{1}{N}\sum_{1\leq i\leq N}Y_i.
		$$
		If there exists a constant $\delta>0$ such that the variance $\mathbb{V}(\overline{R}_N)=O(N^{-\delta})$ for $N\geq1$, then we have
		$$
		\lim_{N\rightarrow\infty}\overline{R}_N=\sigma
		$$
		almost surely.
	\end{lemma}
	
	\section{An arithmetic function} 
	
	Given integers $k\geq 2$, $a,\ b\geq 1$ with $a\neq b$, $r\geq 0$, we define an arithmetic function $f$ by
	
	\begin{align}\label{def: f(i)=}
		f(i)=f_{k,a,b,r}(i):=\sum_{\substack{1\leq d\leq (ai+r)^{1/k}\\ \gcd(a-b,d^k)\mid bi+r}}\frac{\mu(d)\gcd(a-b,d^k)}{d^k},
	\end{align}
	where $\mu$ is the M\"{o}bius function. Trivially, for $a\neq b$ and $k\geq 2$, we have
	
	\begin{align}\label{eq: upper bound of f(i)}
		f(i)=O_{a,b}(1).
	\end{align}
	
	The following lemma gives an asymptotic formula for the the mean value of $f$.
	
	\begin{lemma}\label{lem:mean value of f}
		
		With the above notations, for any $N\geq 1$, we have
		$$
		\sum_{1\leq i\leq N}f(i)=\theta_k(a,b,r)N+O(N^{\frac{1}{k}}),
		$$
		where $\theta_k(a,b,r)$ is given by \eqref{def: theta(k,a,b,r)} and the implied constant in the $O$-symbol depends on $k,\ a,\ b$ and $r$.
	\end{lemma}
	
	\begin{proof} Without loss of generality, we may assume $a>b$. Through out this proof, the implied constants in the symbols $O$ and $\ll$ depend at most on $k,\ a,\ b$ and $r$.
		
		By the definition of $f$, we have
		$$
		\sum\limits_{1\leq i\leq N}f(i)=\sum\limits_{1\leq i\leq N}\sum_{\substack{1\leq d\leq (ai+r)^{1/k}\\ \gcd(a-b,d^k)\mid bi+r}}\frac{\mu(d)\gcd(a-b,d^k)}{d^k}.
		$$
		Changing the order of summations, we obtain
		$$
		\sum_{1\leq i\leq N}f(i)=\sum_{\substack{1\leq d\le (aN+r)^{1/k}}}\frac{\mu(d)\gcd(a-b,d^k)}{d^k}\sum_{\substack{1\leq i\leq N\\ bi\equiv -r\bmod\gcd(a-b,d^k)}}1.
		$$
		Note that the congruence equation
		\[
		bi\equiv -r \bmod \gcd(a-b,d^k)
		\]
		has integer solutions if and only if $\gcd(b,\gcd(a-b,d^k))=\gcd(a,b,d^k)\mid r$, and these solutions forms a residue class
		\[
		i\equiv c \bmod h(d)
		\]
		for some $c=c(a,b,k,d,r)\in\mathbb{Z}$, where
		\[
		h(d)=h(d;k,a,b):=\frac{\gcd(a-b,d^k)}{\gcd(a,b,d^k)}.
		\]
		Then we have
		\begin{align*}
			\sum_{1\leq i\leq N}f(i)=&\sum_{\substack{1\leq d\le (aN+r)^{1/k}\\ \gcd(a,b,d^k)\mid r}}\frac{\mu(d)\gcd(a-b,d^k)}{d^k}\sum_{\substack{1 \leq i\leq N\\i\equiv c \bmod h(d)}}1\\
			=&\sum_{\substack{1\leq d\le (aN+r)^{1/k}\\ \gcd(a,b,d^k)\mid r}}\frac{\mu(d)\gcd(a-b,d^k)}{d^k}\Big(\frac{N}{h(d)}+O(1)\Big).
		\end{align*}
		Since $k\geq 2$, $a>b$ and $|\mu(d)|\leq 1$ for any positive integer $d$, the total contribution of the above $O(1)$ term is
		\[
		\ll \sum\limits_{1\leq d\leq (aN+r)^{1/k}}\frac{1}{d^k}\ll 1.
		\]
		It follows that
		\begin{align*}
			\sum_{1\leq i\leq N}f(i)&=N\sum_{\substack{1\leq d\le (aN+r)^{1/k}\\\gcd(a,b,d^k)\mid r}}\frac{\mu(d)\gcd(a-b,d^k)}{d^kh(d)}+O(1)\\
			&=N\sum_{\substack{1\leq d\le (aN+r)^{1/k}\\\gcd(a,b,d^k)\mid r}}\frac{\mu(d)\gcd(a,b,d^k)}{d^k}+O(1).
		\end{align*}		
		Extending the range of $d$, we obtain
		
		\begin{align*}
			\sum_{1\leq i\leq N}f(i)&=N\sum_{\substack{d=1\\ \gcd(a,b,d^k)\mid r}}^\infty\frac{\mu(d)\gcd(a,b,d^k)}{d^k}+O\Big(N\sum_{d>(Na+r)^{1/k}}\frac{1}{d^k}\Big)+O(1)\\
			&=N\sum_{\substack{d=1\\ \gcd(a,b,d^k)\mid r}}^\infty\frac{\mu(d)\gcd(a,b,d^k)}{d^k}+O(N^{\frac{1}{k}}).
		\end{align*}
		
		Then our desired result follows from the Euler product
		\begin{align*}
			\sum_{\substack{d=1\\ \gcd(a,b,d^k)\mid r}}^\infty\frac{\mu(d)\gcd(a,b,d^k)}{d^k}=\prod_{\substack{p\\ \gcd(a,b,p^k)\mid r}}\Big(1-\frac{\gcd(a,b,p^k)}{p^k}\Big).
		\end{align*}
	\end{proof}

	\section{Proof of Theorem \ref{thm:S_N=}} \label{sec: proof of our theorem}
	
	In this section, we prove Theorem \ref{thm:S_N=} by applying Lemma \ref{lem: the second moment method}. This requires us to compute the expectation $\mathbb{E}(\overline{S}_k(N))$ and estimate the variance $\mathbb{V}(\overline{S}_k(N))$.
	\begin{lemma}\label{lem: E(S_N) and V(S_N)}
		For any integer $N\geq 1$, we have
		\begin{align}\label{eq: E(s)}
			\mathbb{E}(\overline{S}_k(N))=\theta_k(a,b,r)+O\Big(\frac{N^{1/k-1/2}}{\sqrt{\alpha(1-\alpha)}}\Big) 
		\end{align}
		and
		\begin{align}\label{eq: V(s)}
			\mathbb{V}(\overline{S}_k(N))=O\Big(\frac{N^{1/k-1/2}}{{\alpha(1-\alpha)}}\Big),
		\end{align}
		where $\theta_k(a,b,r)$ is the constant defined by \eqref{def: theta(k,a,b,r)}, and the implied constant in the $O$-symbol depends on $k,\ a,\ b$ and $r$.
	\end{lemma}	
	
	Combining Lemma \ref{lem: E(S_N) and V(S_N)} with Lemma \ref{lem: the second moment method}, we obtain Theorem \ref{thm:S_N=}.
	
	Now we only need to prove Lemma \ref{lem: E(S_N) and V(S_N)}. We complete this task in the following three subsections.
	
	\subsection{A sum with $k$-free restriction}\label{subsec: a sum with a k-free restriction}
	
	In this subsection, we prove the following key lemma.
	
	\begin{lemma}\label{lem:sum over k free numers}
		For any integers $n,\ u,\ v\geq 1$ and any $\alpha\in(0,1)$, we have
		$$
		\sum_{\substack{0\leq m\leq n\\ um+v\ {\rm{is}}\ k\text{-}{\rm{free}}}}\binom{n}{m}\alpha^{m}(1-\alpha)^{n-m}= M_k(n,u,v,d)+O\Big(\frac{(un+v)^{1/k}}{\sqrt{\alpha(1-\alpha)n}}\Big),
		$$
		where
		$$
		M_k(n,u,v,d):=\sum_{\substack{1\leq d\leq (un+v)^{1/k}\\ \gcd(u,d^k)\mid v}}\frac{\mu(d)\gcd(u,d^k)}{d^k},
		$$
		and the implied constant in the $O$-symbol is absolute.
	\end{lemma}
	
	To prove Lemma \ref{lem:sum over k free numers}, we first simplify our notations and denote
	$$
	I(n):=\sum_{\substack{0\leq m\leq n\\
			um+v\ {\rm{is}}\ k\text{-}{\rm{free}}}}\binom{n}{m}\alpha^{m}(1-\alpha)^{n-m}.
	$$
	Let $\mu_k$ be the characteristic function of $k$-free numbers. It is well-known that
	$$
	\mu_k(n)=\sum_{d^k\mid n}\mu(d),
	$$
	where $\mu$ is the M\"{o}bius function. Then we have
	$$
	I(n)=\sum_{0\leq m\leq n}\binom{n}{m}\alpha^{m}(1-\alpha)^{n-m}\sum_{d^k\mid um+v}\mu(d).
	$$
	Changing the order of summations, we obtain
	$$
	I(n)=\sum_{1\leq d\leq (un+v)^{1/k}}\mu(d)\sum_{\substack{0\leq m\leq n\\ um\equiv -v\bmod d^k}}\binom{n}{m}\alpha^{m}(1-\alpha)^{n-m}.
	$$
	Observe that the congruence equation
	$$
	um\equiv-v \bmod d^k
	$$
	has integer solutions if and only if $\gcd(u,d^k)\mid v$, in which case, the solutions form a residue class
	$$
	m\equiv c \bmod d^k/\gcd(u,d^k)
	$$
	for some $c=c(u,v,k,d)\in\mathbb{Z}$. Then applying Lemma \ref{lem:sum_{i=r(d^k)}=} to the sum over $m$, we obtain
	
	$$
	I(n)=\sum_{\substack{1\leq d\leq (un+v)^{1/k}\\ \gcd(u,d^k)\mid v}}\mu(d)\bigg(\frac{\gcd(u,d^k)}{d^k}+O\Big(\frac{1}{\sqrt{\alpha(1-\alpha)n}}\Big)\bigg),
	$$
	which implies Lemma \ref{lem:sum over k free numers}.

	\subsection{The expectation}
	In this subsection, our goal is to prove \eqref{eq: E(s)} in Lemma \ref{lem: E(S_N) and V(S_N)} by using Lemma \ref{lem:sum over k free numers}. Without loss of generality, we always assume $a>b$. Through out this subsection, the implied constants in symbols $O$ and $\ll$ depend at most on $k,\ a,\ b$ and $r$.
	
	By the definition of $X_i$, we have
	$$
	\mathbb{E}(X_i)=\mathbb{P}(P_i\ {\rm{is}}\ k\text{-}{\rm{free}}).
	$$
	In the $\alpha$-random walk we consider, the coordinate $P_i$ of the $i$th step is of the form
	$$
	P_i=r+la+(i-l)b
	$$
	for some $0\leq l\leq i$. For each $l$, it is not difficult to see that
	$$
	\mathbb{P}\big(P_i=r+la+(i-l)b\big)=\binom{i}{l}\alpha^{l}(1-\alpha)^{i-l}.
	$$
	Therefore, we have
	\begin{align}\label{eq:E(X_i)=}
		\mathbb{E}(X_i)=\sum_{\substack{0\leq l\leq i\\ (a-b)l+ib+r\ {\rm{is}}\ k\text{-}{\rm{free}}}}\binom{i}{l}\alpha^{l}(1-\alpha)^{i-l}.
	\end{align}
	
	Applying Lemma \ref{lem:sum over k free numers} with $n=i,u=a-b$ and $v=ib+r$, we obtain
	\begin{equation}\label{eq:E(X)}
		\mathbb{E}(X_i)=f(i)+O\Big(\frac{i^{1/k-1/2}}{\sqrt{\alpha(1-\alpha)}} \Big),
	\end{equation}
	where $f(i)=f_{k,a,b,r}(i)$ is defined by \eqref{def: f(i)=}. Summing over $i$ and using \eqref{eq:sum_{s<x}s^{beta}=} to bound the error term, we obtain
	$$
	\mathbb{E}(\overline{S}_k(N))=\frac{1}{N}\sum_{1\leq i\leq N}f(i)+O\Big(\frac{N^{1/k-1/2}}{\sqrt{\alpha(1-\alpha)}} \Big).
	$$
	This combining with Lemma \ref{lem:mean value of f} establishes \eqref{eq: E(s)}.
	\subsection{The variance}
	
	In this subsection, we prove \eqref{eq: V(s)} in Lemma \ref{lem: E(S_N) and V(S_N)}. Again, without loss of generality, we assume $a>b$. Through out this subsection, the implied constants in symbols $O$ and $\ll$ depend at most on $k,\ a,\ b$ and $r$.
	
	Note that the variance
	\begin{align}\label{eq:V(s(n))=E(s(n))^2-(E(s(n)))^2}
		\mathbb{V}(\overline{S}_k(N))=\mathbb{E}\big(\overline{S}_k(N)^2\big)-\mathbb{E}(\overline{S}_k(N))^2.
	\end{align}
	For $\mathbb{E}(\overline{S}_k(N))^2$, it follows from \eqref{eq: E(s)} and the estimate $\sqrt{\alpha(1-\alpha)}\geq \alpha(1-\alpha)$ that
	\begin{align}\label{eq:(E(s(n)))^2=}
		\mathbb{E}(\overline{S}_k(N))^2=\theta(k,a,b,r)^2+O\Big(\frac{N^{1/k-1/2}}{{\alpha(1-\alpha)}}\Big).		\end{align}
	To deal with $\mathbb{E}\big(\overline{S}_k(N)^2\big)$, by the definition of $\overline{S}_k(N)$, we write
	
	\begin{align}\label{eq:E((s(n))^2)=sum_E(X_iX_j)+sum_E(X_i^2)}
		\mathbb{E}\big(\overline{S}_k(N)^2\big)=\frac{2}{N^2}\sum_{1\leq i<j\leq N}\mathbb{E}(X_iX_j)+\frac{1}{N^2}\sum_{1\leq i\leq N}\mathbb{E}({X_i}^2).
	\end{align}

	For the second term on the right hand side of \eqref{eq:E((s(n))^2)=sum_E(X_iX_j)+sum_E(X_i^2)}, we have
	\begin{equation}\label{eq:E(X^2)}
		\frac{1}{N^2}\sum_{1\leq i\leq N}\mathbb{E}({X_i}^2)=\frac{1}{N^2}\sum_{1\leq i\leq N}\mathbb{E}(X_i)=O(N^{-1}),
	\end{equation}
	which is acceptable.
	
	Now we consider the first term on the right hand side of \eqref{eq:E((s(n))^2)=sum_E(X_iX_j)+sum_E(X_i^2)}. Note that
	$$
	\mathbb{E}(X_iX_j)=\mathbb{P}(P_i\ {\rm{and}}\ P_j\ {\rm{are\ both}}\ k\text{-}{\rm{free}}).
	$$
	For $i<j$, the probability that $P_i=r+la+(i-l)b$ and $P_j=P_i+ha+(j-i-h)b$ for some $0\leq l\leq i$ and $0\leq h\leq j-i$ is equal to
	$$
	\binom{i}{l}\alpha^l(1-\alpha)^{i-l}\binom{j-i}{h}\alpha^h(1-\alpha)^{j-i-h}.
	$$
	Hence, similar to \eqref{eq:E(X_i)=}, we have
	$$
	\mathbb{E}(X_iX_j)=\sum_{\substack{0\leq l\leq i\\ r+(a-b)l+bi\ {\rm{is}}\ k\text{-}{\rm{free}}}}\binom{i}{l}\alpha^l(1-\alpha)^{i-l}\sum_{\substack{0\leq h\leq j-i\\ (a-b)h+(a-b)l+bj+r\ {\rm{is}}\ k\text{-}{\rm{free}}}}\binom{j-i}{h}\alpha^h(1-\alpha)^{j-i-h}.
	$$
	
	Applying Lemma \ref{lem:sum over k free numers} with $n=j-i,\ u=a-b$ and $v=(a-b)l+bj+r$, the inner sum over $h$ is equal to
	$$
	\sum\limits_{\substack{1\leq d\le (aj+r-(a-b)(i-l))^{1/k}\\ \gcd(a-b,d^k)\mid bj+r}}\frac{\mu(d)\gcd(a-b,d^k)}{d^k}+O\Big(\frac{({j-i})^{-1/2}j^{1/k}}{\sqrt{\alpha(1-\alpha)}}\Big).
	$$
	Note that
	\begin{align*}
		(aj+r-(a-b)(i-l))^{1/k}&=(aj+r)^{1/k}\bigg(1+O\Big(\frac{(a-b)(i-l)}{aj+r}\Big)\bigg)\\
		&=(aj+r)^{1/k}\big(1+O(ij^{-1})\big).
	\end{align*}
	Then we have
	$$
	\sum\limits_{\substack{(aj+r-(a-b)(i-l))^{1/k}<d\le (aj+r)^{1/k}\\ \gcd(a-b,d^k)\mid bj+r}}\frac{\mu(d)\gcd(a-b,d^k)}{d^k}\ll ij^{-2+1/k}.
	$$
	Thus, the inner sum over $h$ is equal to
	\[
	f(j)+O\Big(ij^{-2+1/k}+\frac{({j-i})^{-1/2}j^{1/k}}{\sqrt{\alpha(1-\alpha)}}\Big).
	\]
	It follows that
	\begin{align*}
		\mathbb{E}(X_iX_j)&=\sum_{\substack{0\leq l\leq i\\ r+(a-b)l+bi\ {\rm{is}}\ k\text{-}{\rm{free}}}}\binom{i}{l}\alpha^l(1-\alpha)^{i-l}\bigg(f(j)+O\Big(ij^{-2+1/k}+\frac{({j-i})^{-1/2}j^{1/k}}{\sqrt{\alpha(1-\alpha)}}\Big)\bigg)\\
		&=f(j)\sum_{\substack{0\leq l\leq i\\ (a-b)l+bi+r\ {\rm{is}}\ k\text{-}{\rm{free}}}}\binom{i}{l}\alpha^l(1-\alpha)^{i-l}+O\Big(ij^{-2+1/k}+\frac{({j-i})^{-1/2}j^{1/k}}{\sqrt{\alpha(1-\alpha)}}\Big).
	\end{align*}
	
	Applying Lemma \ref{lem:sum over k free numers} again with $n=i$, $u=a-b$ and $v=bi+r$, we obtain
	
	\begin{align*}
		\mathbb{E}(X_i X_j)=f(i)f(j)+O\Big(\frac{i^{1/k-1/2}}{\sqrt{\alpha(1-\alpha)}}\Big)+O\Big(ij^{-2+1/k}+\frac{({j-i})^{-1/2}j^{1/k}}{\sqrt{\alpha(1-\alpha)}}\Big),
	\end{align*}
	where we have used the bound \eqref{eq: upper bound of f(i)}. Taking sum over $1\leq i<j\leq N$ and applying Corollary \ref{cor: sums over i,j} to bound the contribution of the $O$-term, we derive
	$$
	\sum_{1\leq i<j\leq N}\mathbb{E}(X_iX_j)=\sum_{1\leq i<j\leq N}f(i)f(j)+O\Big(\frac{N^{3/2+1/k}}{\sqrt{\alpha(1-\alpha)}}\Big).
	$$
	Adding the diagonal terms up to an error term
	$$
	\sum_{1\leq i\leq N}f(i)^2\ll\sum_{1\leq i\leq N}1\ll N,
	$$
	we then obtain
	$$
	\sum_{1\leq i<j\leq N}\mathbb{E}(X_iX_j)=\frac{1}{2}\Big(\sum_{1\leq i\leq N}f(i)\Big)^2+O\Big(\frac{N^{3/2+1/k}}{\sqrt{\alpha(1-\alpha)}}\Big).
	$$
	It is readily to seen from Lemma \ref{lem:mean value of f} that 
	\begin{equation}\label{eq:sum_{i,j}E(X_iX_j)=}
		\sum_{1\leq i<j\leq N}\mathbb{E}(X_iX_j)=\frac{N^2}{2}\theta_k(a,b,r)^2+O\Big(\frac{N^{3/2+1/k}}{\sqrt{\alpha(1-\alpha)}}\Big),
	\end{equation}
	where $\theta_k(a,b,r)$ is defined in \eqref{def: theta(k,a,b,r)}.
	
	Combining \eqref{eq:E(X^2)} and \eqref{eq:sum_{i,j}E(X_iX_j)=} with \eqref{eq:E((s(n))^2)=sum_E(X_iX_j)+sum_E(X_i^2)}, we obtain
	
	$$
	\mathbb{E}\big(\overline{S}_k(N)^2\big)=\theta_k(a,b,r)^2+O\Big(\frac{N^{1/k-1/2}}{\sqrt{\alpha(1-\alpha)}}\Big).
	$$
	Plugging this and \eqref{eq:(E(s(n)))^2=} into \eqref{eq:V(s(n))=E(s(n))^2-(E(s(n)))^2}, we obtain our desired result \eqref{eq: V(s)}.

\end{document}